\documentclass{article}

\usepackage{arxiv}

\usepackage[utf8]{inputenc} 
\usepackage[T1]{fontenc}    

\usepackage{hyperref}       
\usepackage{url}            
\usepackage{booktabs}       
\usepackage{amsfonts}       
\usepackage{nicefrac}       
\usepackage{microtype}      
\usepackage{graphicx}
\usepackage{natbib}
\usepackage{doi}

\usepackage{algorithm}

\usepackage{amsopn}

\usepackage{mathtools} 
\usepackage{amssymb} 
\usepackage{mathabx} 
\usepackage{siunitx}
\usepackage{mathscinet}

\usepackage{esint} 
\usepackage[bb=dsserif]{mathalpha} 

\usepackage{isomath}
\usepackage{algorithmic} 
\usepackage{textcomp} 

\usepackage{caption}
\usepackage{subcaption}

\newcommand{\Z}{\mathbb{Z}}
\newcommand{\R}{\mathbb{R}}
\newcommand{\Cfield}{\mathbb{C}}
\newcommand{\I}{\mathcal{I}}
\newcommand{\Lop}{\mathcal{L}}

\newcommand{\TDop}{\tilde{\mathcal{D}}}
\newcommand{\DopS}{\mathcal{D}^\textnormal{s}}
\newcommand{\TDopS}{\tilde{\mathcal{D}}^\textnormal{s}}
\newcommand{\Q}{\mathcal{Q}}
\newcommand{\Proj}{\mathcal{P}}

\newcommand{\vecx}{\vectorsym{x}}
\newcommand{\vecZero}{\vectorsym{0}}
\newcommand{\vecOne}{\vectorsym{1}}
\newcommand{\vecxi}{\vectorsym{\xi}}
\newcommand{\vecI}{\vectorsym{I}}
\newcommand{\vecu}{\vectorsym{u}}
\newcommand{\vecups}{\vectorsym{\upsilon}}
\newcommand{\vecv}{\vectorsym{v}}
\newcommand{\veck}{\vectorsym{k}}
\newcommand{\vecb}{\vectorsym{b}}
\newcommand{\vecw}{\vectorsym{w}}
\newcommand{\matE}{\matrixsym{E}}
\newcommand{\matF}{\matrixsym{F}}
\newcommand{\matB}{\matrixsym{B}}
\newcommand{\matU}{\matrixsym{U}}
\newcommand{\matV}{\matrixsym{V}}
\newcommand{\matVarep}{\matrixsym{\varepsilon}}
\newcommand{\matep}{\matrixsym{\epsilon}}
\newcommand{\mattau}{\matrixsym{\tau}}
\newcommand{\matsigma}{\matrixsym{\sigma}}
\newcommand{\tenC}{\tensorsym{C}}
\newcommand{\tenGamma}{\tensorsym{\Gamma}}
\newcommand{\iu}{\mathrm{i}\mkern1mu}
\newcommand{\nablaS}{\nabla^\textnormal{s}}
\newcommand{\otimesS}{\otimes^\textnormal{s}}
\newcommand{\Brackets}[1]{\left( #1 \right)}
\newcommand{\SquareBrackets}[1]{\left[ #1\right]}
\newcommand{\Braces}[1]{\left\{ #1\right\}}
\newcommand{\AngleBrackets}[1]{\left\langle #1 \right\rangle}

\newcommand{\Norm}[1]{\left\lVert #1 \right\rVert}
\newcommand{\Seminorm}[1]{\left\lvert #1 \right\rvert}
\newcommand{\vertiii}[1]{{\left\vert\kern-0.25ex\left\vert\kern-0.25ex\left\vert #1\right\vert\kern-0.25ex\right\vert\kern-0.25ex\right\vert}}
\newcommand{\dx}{\,\mathrm{d}}
\newcommand{\Sym}{\mathbb{S}}

\DeclareMathOperator{\Card}{card}
\DeclareMathSymbol{\shortminus}{\mathbin}{AMSa}{"39}

\DeclareMathOperator{\FFT}{FFT}

\usepackage{amsthm}
\theoremstyle{plain}
\newtheorem{theorem}{Theorem}

\newtheorem{lemma}{Lemma}

\newtheorem{proposition}{Proposition}
\newtheorem{definition}{Definition}
\newtheorem{remark}{Remark}

\usepackage{cleveref}

\title{Numerical analysis of several FFT-based schemes for computational homogenization}

\author{ \href{https://orcid.org/0000-0002-0085-6699}{\includegraphics[scale=0.06]{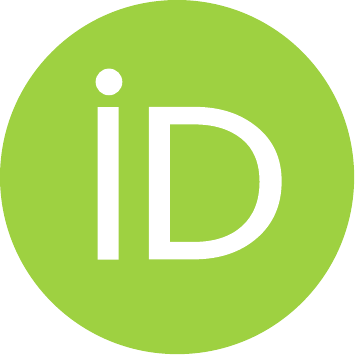}\hspace{1mm}Changqing Ye}\\
	Department of Mathematics\\
	The Chinese University of Hong Kong\\
	Hong Kong Special Administrative Region \\
	\texttt{cqye@math.cuhk.edu.hk} \\
	\And
	Eric T. Chung \\
	Department of Mathematics\\
	The Chinese University of Hong Kong\\
	Hong Kong Special Administrative Region \\
	\texttt{tschung@math.cuhk.edu.hk} \\
}

\renewcommand{\shorttitle}{Numerical analysis of FFT-based homogenization schemes}

\hypersetup{
pdftitle={Numerical analysis of several FFT-based schemes for computational homogenization},
pdfsubject={74Q05, 74Q15, 65N12},
pdfauthor={Changqing YE},
pdfkeywords={Computational homogenization, FFT, effective coefficients, convergence, convergence rates, numerical analysis},
}

\begin{document}
\maketitle

\begin{abstract}
We study the convergences of several FFT-based schemes that are widely applied in computational homogenization for deriving effective coefficients, and the term ``convergence'' here means the limiting behaviors as spatial resolutions going to infinity. Those schemes include Moulinec-Suquent's scheme [Comput Method Appl M, 157 (1998), pp. 69-94], Willot's scheme [Comptes Rendus M\'{e}canique, 343 (2015), pp. 232-245], and the FEM scheme [Int J Numer Meth Eng, 109 (2017), pp. 1461-1489]. Under some reasonable assumptions, we prove that the effective coefficients obtained by those schemes are all convergent to the theoretical ones. Moreover, for the FEM scheme, we can present several convergence rate estimates under additional regularity assumptions.
\end{abstract}


\section{Introduction and preliminaries} \label{sec:Intro}

In an inhomogeneous material that obeys the linear elasticity law, the fourth-order stiffness tensor $\tenC(\vecx)$ varies at different $\vecx$. The central goal of computational homogenization methods is determining an effective constitute law, which may serve as a surrogate model for downstream applications \cite{Milton2002,Zohdi2008}.

Let $Y=(0,1)^d$ be a Representative Volume Element (RVE) where $d=3$ through the whole paper, and $\tenC(\vecx)$ be the fourth-order stiffness tensor for any $\vecx \in Y$. We denote by $H^1_\#(Y)$ the closure of smooth $Y$-periodic functions with respect to ordinary $H^1$-norm \cite{Cioranescu1999}, and define $W^{1,p}_\#(Y)$ accordingly. The notation $H^1_\#(Y)$ will be replaced as $H^1_\#(Y;\R^d)$ or $H^1_\#(Y;\Cfield^d)$ for vector-valued functions, while $\Norm{u}_{W^{m,p}(\omega)}$ ($\Norm{u}_{L^p(\omega)}$, $\Norm{u}_{H^m(\omega)}$) should be always understood as $W^{m,p}$-norm ($L^p$-norm, $H^m$-norm) of $u$ on the domain $\omega$ regardless of whether $u$ is vector-valued or not. For a discrete set $D\subset \Z^d$, let $l(D;W)$ be the linear space of discrete functions which take inputs in $D$ and return values of $W$, where $W$ may be $\R$, $\Cfield$, $\R^d$, $\Cfield^d$ and $\Sym^d$ (the set of symmetric $d\times d$ matrices). The double dot operator ``$:$'' and the single dot operator ``$\cdot$'' are defined conventionally for any order tensors (including vectors and matrices), and note that now the matrix-vector product $\matE\vecu$ is rewritten as $\matE \cdot \vecu$. A popular method to derive the effective coefficients $\tenC^\textnormal{eff}$ of $\tenC(\vecx)$ is solving a periodic boundary value problem: find $\vecu \in H^1_\#(Y;\R^d)$ with $\AngleBrackets{\vecu}=\vecZero$ such that for any $\vecv \in H^1_\#(Y;\R^d)$
\begin{equation} \label{eq:varia form}
\AngleBrackets{\nablaS \vecv : \tenC(\vecx) : \nablaS \vecu} = -\AngleBrackets{\nablaS \vecv : \tenC(\vecx)} : \matE,
\end{equation}
where $\AngleBrackets{f}$ stands for $\fint_Y f \dx \vecx$, $\matE$ belongs to $\Sym^d$ and $\SquareBrackets{\nablaS \vecv}_{mn}=\frac{1}{2}\Brackets{\partial_m \SquareBrackets{\vecv}_n+\partial_n \SquareBrackets{\vecv}_m}$. The effective stiffness tensor $\tenC^\textnormal{eff}$ can be obtained by
\begin{equation} \label{eq:eff}
\tenC^\textnormal{eff}:\matE = \AngleBrackets{\tenC(\vecx):\Brackets{\nablaS \vecu + \matE}}
\end{equation}
via choosing different $\matE$.

The starting point of FFT-based homogenization is rewriting the above variational form into the Lippmann-Schwinger equation. Taking $\tenC^\textnormal{ref}$ as a linear elastic reference medium, which is a \emph{constant} fourth-order tensor satisfying certain coercive conditions, we can introduce a Green function $\tensorsym{G}^0$ such that for any $\matrixsym{F}(\vecx) \in L^2(Y;\Sym^d)$, the convolution $\tensorsym{G}^0 * \matrixsym{F}\in H^1_\#(Y;\R^d)$ is the solution to the variational form
\[
\AngleBrackets{\nablaS \vecv : \tenC^\textnormal{ref} : \nablaS \Brackets{\tensorsym{G}^0*\matrixsym{F}}} = \AngleBrackets{\nablaS \vecv : \matrixsym{F}}, \forall \vecv \in H^1_\#(Y;\R^d).
\]
By setting $[\tenC^\textnormal{ref}]_{mnpq}=\lambda^0\delta_{mn}\delta_{pq}+\mu^0(\delta_{mp}\delta_{nq}+\delta_{mq}\delta_{np})$ and taking the Fourier series expansion of $\matrixsym{F}(\vecx)$ as $\matrixsym{F}(\vecx)=\sum_{\vecxi \in \Z^d}\widehat{\matrixsym{F}}\exp(2\pi \iu  \vecx \cdot \vecxi)$, we can derive that for $\vecxi\neq \vectorsym{0}$
\[
\widehat{\tensorsym{G}^0*\matrixsym{F}}\SquareBrackets{\vecxi}=\frac{\iu }{2\pi\Seminorm{\vecxi}^2\mu^0} \SquareBrackets{\frac{\Brackets{\mu^0+\lambda^0}\vecxi\otimes \vecxi }{\Brackets{2\mu^0+\lambda^0} \Seminorm{\vecxi}^2}-\matrixsym{I}_d}  \cdot \widehat{\matrixsym{F}} \cdot \vecxi,
\]
where $\matrixsym{I}_d$ is the $d\times d$ identity matrix and $\otimes $ is the Kronecker product. Moreover, it is easy to show that \cref{eq:varia form} is equivalent to
\[
\vecu = -\tensorsym{G}^0*\Brackets{\delta \tenC:\Brackets{\nablaS \vecu+\matE}+\tenC^\textnormal{ref}:\matE} = -\tensorsym{G}^0*\Brackets{\delta \tenC:\Brackets{\nablaS \vecu+\matE}},
\]
where $\delta \tenC \coloneqq \tenC-\tenC^\textnormal{ref}$ and $\tensorsym{G}^0*(\tenC^\textnormal{ref}:\matE)$ vanishes due to that $\tenC^\textnormal{ref}:\matE$ is a constant matrix. If taking $\matVarep\coloneqq \nablaS \vecu+\matE$ and $\tensorsym{\Gamma}^0\coloneqq \nablaS \tensorsym{G}^0$, we can immediately read the above equation as
\begin{equation} \label{eq:LS}
\matVarep+\tensorsym{\Gamma}^0*\Brackets{\delta \tenC:\matVarep} = \matE
\end{equation}
the so-called Lippmann-Schwinger equation \cite{Lippmann1950}. The alternative form of \cref{eq:LS} in the Fourier space is
\[
\widehat{\matVarep}\SquareBrackets{\vecxi}=\left\{
\begin{aligned}
&-\widehat{\tensorsym{\Gamma}^0} : \widehat{\delta\tenC : \matVarep}, &\vecxi\neq \vectorsym{0}, \\
&\matE, &\vecxi = \vectorsym{0}.
\end{aligned}
\right.
\]
Specifically, when $[\tenC^\textnormal{ref}]_{mnpq}=\lambda^0\delta_{mn}\delta_{pq}+\mu^0(\delta_{mp}\delta_{nq}+\delta_{mq}\delta_{np})$, we can derive a closed form of $\widehat{\tensorsym{\Gamma}^0}$ as
\[
\begin{aligned}
[\widehat{\tensorsym{\Gamma}^0}]_{mnpq}[\vecxi] = &\frac{1}{4\Seminorm{\vecxi}^2\mu^0}\Brackets{\vecxi_m\vecxi_q\delta_{np}+\vecxi_n\vecxi_q\delta_{mp}+\vecxi_m\vecxi_p\delta_{nq}+\vecxi_n\vecxi_p\delta_{mq}}\\
&-\frac{\Brackets{\lambda^0+\mu^0}\vecxi_m\vecxi_n\vecxi_p\vecxi_q}{\mu^0\Brackets{\lambda^0+2\mu^0}\Seminorm{\vecxi}^4}
\end{aligned}
\]

Now it comes to discretization schemes. For simplicity, we equally split $Y$ into $N$ parts in every dimension. For any $d$-dimensional index $\vecI=(\vecI_1,\vecI_2,\vecI_3) \in \I_N \subset \mathbb{Z}^d$ with $\I_N=\Braces{\vecI\in \Z^d: 0\leq \vecI_m < N,\forall 1\leq m \leq d}$, we denote by
\[
Y_\vecI=\Brackets{\frac{\vecI_1}{N}, \frac{\vecI_1+1}{N}}\times \Brackets{\frac{\vecI_2}{N}, \frac{\vecI_2+1}{N}}\times \Brackets{\frac{\vecI_3}{N}, \frac{\vecI_3+1}{N}}
\]
an element in the language of Finite Element Methods (FEM), which could be interpreted as a pixel (for 2D problems) or voxel (for 3D problems). Correspondingly, the frequency domain is denoted by $\mathcal{F}_N\coloneqq\Braces{\vecxi\in \Z^d:-N/2\leq \vecxi_m<N/2,\forall 1\leq m \leq d}$. We always postulate that an image-like $\tenC_N(\vecx)$ (i.e., $\tenC_N(\vecx)$ is constant in every $Y_\vecI$) rather that the ground truth $\tenC(\vecx)$ is provided, and introduce a notation $\tenC_N^\star[\vecI]=\tenC_N(\vecx_\vecI)$ where $\vecx_\vecI$ is the geometric center of $Y_\vecI$. The basic scheme of FFT-based homogenization was proposed by Moulinec and Suquent in \cite{Moulinec1995,Moulinec1998}, and can be stated as \cref{alg:MS}. Depending on the type of a function's domain, the Fourier coefficients are defined differently, i.e. if $f\in L^2(Y;\Cfield)$, then $\widehat{f}[\vecxi]$ is determined by the formula
\[
f(\vecx) = \sum_{\vecxi \in \Z^d} \widehat{f}[\vecxi]\exp(2\pi\iu \vecxi\cdot\vecx),
\]
and if $f^\star\in l(\I_N;\mathbb{C})$, then $\widehat{f^\star}[\vecxi]$ satisfies
\[
f^\star[\vecI] = \sum_{\vecxi\in \mathcal{F}_N} \widehat{f^\star}[\vecxi] \exp(2\pi \iu \frac{\vecxi\cdot \vecI}{N}).
\]
Note that if is labeled with a superscript ``$\star$'', the function should be understood as a discrete one with $\I_N$ as its domain.

\begin{algorithm}
\caption{Moulinec-Suquent's scheme}
\label{alg:MS}
\begin{algorithmic}[1]
\STATE{Set $\tenC^\textnormal{ref}$ and $\matE$, initiate a tensor variable $\matVarep_N^\star \in l(\I_N;\Sym^d)$ with $\matVarep_N^\star[\vecI]=\matE$ for all $\vecI \in \I_N$}
\WHILE{not meet the convergence criterion}
\STATE{Evaluate $\matrixsym{\tau}_N^\star[\vecI] = \Brackets{\tenC_N^\star[\vecI]-\tenC^\textnormal{ref}}:\matVarep_N^\star[\vecI]$ for all $\vecI \in \I_N$}
\STATE{Perform a FFT $\widehat{\matrixsym{\tau}_N^\star}=\FFT\Brackets{\matrixsym{\tau}_N^\star}$}
\STATE{Initiate a temporary tensor variable $\widehat{\matep_N^\star}$ and evaluate $\widehat{\matep_N^\star}[\vecZero]=\matE$, $\widehat{\matep_N^\star}[\vecxi]=-\widehat{\tensorsym{\Gamma}^0}[\vecxi]:\widehat{\matrixsym{\tau}_N^\star}[\vecxi]$ for all $\vecxi \in  \mathcal{F}_N\setminus \Braces{\vecZero}$}
\STATE{Perform an inverse FFT $\matep_N^\star=\FFT^{-1}\Brackets{\widehat{\matep_N^\star}}$}
\STATE{Calculate the convergence indicator via $\matVarep_N^\star$ and $\matep_N^\star$ and $\matVarep_N^\star \leftarrow \matep_N^\star$}
\ENDWHILE
\RETURN $\matVarep_N^\star$ and $N^{-d}\sum_{\vecI\in \I_N} \tenC_N^\star[\vecI]:\matVarep_N^\star[\vecI]$
\end{algorithmic}
\end{algorithm}

Moulinec-Suquent's scheme (we also call it the basic scheme) gains considerable popularity. The reasons can be summarized as follows: 1) microstructures in heterogeneous materials may be rather complex, which causes that generating high quality meshes dominates overall computational overheads \cite{Hughes2005}; 2) the primary information of microstructures is usually provided by modern digital volume imaging techniques \cite{Larson2002,Poulsen2004,Landis2010}, and we may not be capable to retrieve the original geometric descriptions and to perform a preprocessing for FEMs; 3) the easy implementation of those schemes and highly optimized FFT packages (e.g., Intel\circledR MKL, FFTW \cite{Frigo2005}) secure the global efficiency; 4) generally, it is $\tenC^\textnormal{eff}$ the effective coefficients we are more interested in rather than $\vecu$ the local displacement or $\matVarep$ the local stress, and unfitted structured meshes may have less influence on desired $\tenC^\textnormal{eff}$ comparing to $\vecu$ and $\matVarep$. Moreover, those methods are not limited in linear elasticity, versatile applications such as nonlinear elasticity \cite{Moulinec1998,Michel2001},  piezoelectricity \cite{Brenner2009}, damage and fracture mechanics \cite{Zhu2015,Chai2020}, polycrystalline materials \cite{Segurado2018} can be found in literature. We strongly recommend a recent review article \cite{Schneider2021} as an exhaustive reference for historical developments and the current state of the art of FFT-base homogenization methods.


Willot's scheme is another popular discretization method of the Lippmann-Schwinger equation \cite{Willot2015}. The main idea in Willot's scheme is replacing $2\pi \iu \vecxi$ the gradient operator $\nabla$ in the \emph{Fourier space} with
\begin{equation}\label{eq:WL k_N}
\veck_N[\vecxi]=\frac{\iu N}{4}\prod_{m=1}^d\Braces{\exp\Brackets{2\pi\iu\frac{\vecxi_m}{N}}+1}\SquareBrackets{\tan\Brackets{\frac{\pi\vecxi_1}{N}},\tan\Brackets{\frac{\pi\vecxi_2}{N}},\tan\Brackets{\frac{\pi\vecxi_3}{N}}}
\end{equation}
that is derived from a finite difference stencil. If $\tenC^\textnormal{ref}$ is isotropic with the Lam\'{e} coefficients $(\lambda^0,\mu^0)$, the closed form of $\widehat{\tensorsym{\Gamma}^0_\textnormal{W}}$ is
\[
\begin{aligned}
[\widehat{\tensorsym{\Gamma}^0_\textnormal{W}}]_{mnpq}[\vecxi]=&\frac{1}{4\mu^0\Seminorm{\veck_N}^2}\left\{ \SquareBrackets{\veck_N}_m\SquareBrackets{\widebar{\veck_N}}_q\delta_{np}+\SquareBrackets{\veck_N}_n\SquareBrackets{\widebar{\veck_N}}_q\delta_{mp}\right. \\
\\&\qquad\qquad\left.+\SquareBrackets{\veck_N}_m\SquareBrackets{\widebar{\veck_N}}_p\delta_{nq}+\SquareBrackets{\veck_N}_n\SquareBrackets{\widebar{\veck_N}}_p\delta_{mq}\right\}\\
&-\frac{\Brackets{\lambda^0+\mu^0}\SquareBrackets{\veck_N}_m\SquareBrackets{\widebar{\veck_N}}_n\SquareBrackets{\veck_N}_p\SquareBrackets{\widebar{\veck_N}}_q}{\mu^0\Brackets{\lambda^0+2\mu^0}\Seminorm{\veck_N}^4},
\end{aligned}
\]
and this is the major difference from the basic scheme. Take
\[
\mathcal{F}_{N\shortminus}\coloneqq\Braces{\vecxi\in\Z^d:-N/2<\vecxi_m<N/2,\forall 1\leq m\leq d},
\]
and Willot's scheme states as \cref{alg:WL}.

\begin{algorithm}
\caption{Willot's scheme}
\label{alg:WL}
\begin{algorithmic}[1]
\STATE{Set $\tenC^\textnormal{ref}$ and $\matE$, initiate a tensor variable $\matVarep_N^\star \in l(\I_N;\Sym^d)$ with $\matVarep_N^\star[\vecI]=\matE$ for all $\vecI \in \I_N$}
\WHILE{not meet the convergence criterion}
\STATE{Evaluate $\matrixsym{\tau}_N^\star[\vecI] = \Brackets{\tenC_N^\star[\vecI]-\tenC^\textnormal{ref}}:\matVarep_N^\star[\vecI]$ for all $\vecI \in \I_N$}
\STATE{Perform a FFT $\widehat{\matrixsym{\tau}_N^\star}=\FFT\Brackets{\matrixsym{\tau}_N^\star}$}
\STATE{Initiate a temporary tensor variable $\widehat{\matep_N^\star}$, evaluate $\widehat{\matep_N^\star}[\vecZero]=\matE$, $\widehat{\matep_N^\star}[\vecxi]=-\widehat{\tensorsym{\Gamma}^0_\textnormal{W}}[\vecxi]:\widehat{\matrixsym{\tau}_N^\star}[\vecxi]$ for all $\vecxi \in  \mathcal{F}_{N\shortminus}\setminus \Braces{\vecZero}$, and $\widehat{\matep_N^\star}[\vecxi]=\vecZero$ for all $\vecxi\in\mathcal{F}_N\setminus \mathcal{F}_{N\shortminus}$}
\STATE{Perform an inverse FFT $\matep_N^\star=\FFT^{-1}\Brackets{\widehat{\matep_N^\star}}$}
\STATE{Calculate the convergence indicator via $\matVarep_N^\star$ and $\matep_N^\star$ and $\matVarep_N^\star \leftarrow \matep_N^\star$}
\ENDWHILE
\RETURN $\matVarep_N^\star$ and $N^{-d}\sum_{\vecI\in \I_N} \tenC_N^\star[\vecI]:\matVarep_N^\star[\vecI]$
\end{algorithmic}
\end{algorithm}

It has been numerically demonstrated that convergences of Moulinec-Suquent's scheme will deteriorate as the contrast ratio of $\tenC_N$ grows, and an extreme example shows that the basic scheme indeed fails to converge for a porous material \cite{Schneider2016}, while Willot's scheme stays a stable convergence history in high contrast settings \cite{Willot2015}. Meanwhile, computational overheads of the basic and Willot's schemes are almost same. Those advantages make Willot's scheme become another standard method in discretizing the Lippmann-Schwinger equation \cref{eq:LS}.

Interestingly, it is pointed in \cite{Schneider2017} that Willot's scheme is related to a \emph{reduced integration} variational problem: find $\vecu_N \in V_N^d$ with $\AngleBrackets{\vecu_N}=\vecZero$ such that for any $\vecv_N \in V_N^d$,
\[
N^{-d}\sum_{\vecI\in \I_N} \nablaS \vecv_N(\vecx_\vecI):\tenC_N^\star[\vecI]:\nablaS \vecu_N(\vecx_\vecI)=-N^{-d}\sum_{\vecI\in \I_N} \nablaS \vecv_N(\vecx_\vecI):\tenC_N^\star[\vecI]:\matE,
\]
where $V_N$ is the trilinear Lagrange finite element space of $H^1_\#(Y)$. The reduced integration technique is significantly efficient in constructing stiffness matrices comparing to \emph{full integration}, while the latter needs evaluations on all the Gaussian quadrature points ($8$ for trilinear elements) in each element. However, due to that the stiffness matrix derived by reduced integration may be singular, there exist so-called ``hourglassing'' or ``zero-energy'' modes \cite{Koh1987,Pugh1978}. For example when $N$ is an even integer, it is easy to find nonzero $\vecv_N$ such that $\nablaS \vecv_N(\vecx_\vecI)=\vecZero$ for all $\vecI$, and this $\vecv_N$ is closely related with the frequency $\vecxi=[-N/2,-N/2,-N/2]$, which also implies there does \emph{not} exist a positive constant $c$ independent of $N$ satisfying
\[
c\Norm{\vecv_N}_{H^1(Y)}^2\leq N^{-d}\sum_{\vecI\in \I_N} \nablaS \vecv_N(\vecx_\vecI):\tenC_N^\star[\vecI]:\nablaS \vecv_N(\vecx_\vecI).
\]

In \cite{Schneider2017}, Schneider et al. proposed a novel scheme by replacing reduced integration in Willot's scheme with full integration, i.e., find $\vecu_N \in V_N^d$ with $\AngleBrackets{\vecu_N}=\vecZero$ such that for any $\vecv_N \in V_N^d$,
\[
\begin{aligned}
&(2N)^{-d}\sum_{\vecI\in \I_N}\sum_{\vecb\in \Braces{-1,1}^d} \nablaS \vecv_N(\vecx_\vecI^\vecb):\tenC_N^\star[\vecI]:\nablaS \vecu_N(\vecx_\vecI^\vecb)=\AngleBrackets{\nablaS \vecv_N : \tenC_N : \nablaS \vecu_N}\\
=&-N^{-d}\sum_{\vecI\in \I_N} \nablaS \vecv_N(\vecx_\vecI):\tenC_N^\star[\vecI]:\matE= -\AngleBrackets{\nablaS \vecv_N : \tenC_N} : \matE,
\end{aligned}
\]
where $\vecx_\vecI^\vecb$ is the Gaussian quadrature point on the element $Y_\vecI$ determined by symbols of $\vecb \in \Braces{-1,1}^d$. The FFT technique is utilized in solving reference problem: find $\vecw_N \in V_N^d$ with $\AngleBrackets{\vecw_N}=\vecZero$ such that for any $\vecv_N \in V_N^d$,
\[
\begin{aligned}
&(2N)^{-d}\sum_{\vecI\in \I_N}\sum_{\vecb\in \Braces{-1,1}^d} \nablaS \vecv_N(\vecx_\vecI^\vecb):\tenC^\textnormal{ref}:\nablaS \vecu_N(\vecx_\vecI^\vecb) = \AngleBrackets{\nablaS \vecv_N: \tenC^\textnormal{ref}:\nablaS \vecw_N} \\
=& (2N)^{-d}\sum_{\vecI\in \I_N}\sum_{\vecb\in \Braces{-1,1}^d} \nablaS \vecv_N(\vecx_\vecI^\vecb):\matF_N^{\vecb,\star}[\vecI],
\end{aligned}
\]
where $\matF_N^{\vecb,\star}[\vecI] \in \Sym^d$. The motivation is converting $\nablaS \vecv_N(\vecx_\vecI^\vecb)$ into $\veck_N^\vecb[\vecxi] \otimesS \widehat{\vecv_N^\star}[\vecxi]$ via the Discrete Fourier Transform (DFT), where $\otimesS$ is the symmetric Kronecker product and $\vecv_N^\star[\vecI]=\vecv_N(\vecI/N)$. Combining the symmetric relations of $\tenC^\textnormal{ref}$, we can hence derive an explicit formula of $\vecw_N(\vecI/N)$ in the Fourier space as $\matB_N^{-1}[\vecxi]\cdot \vectorsym{\zeta}_N[\vecxi]$ for $\vecxi\neq \vecZero$, where $\matB_N=2^{-d}\sum_\vecb \widebar{\veck_N^\vecb} \cdot \tenC^\textnormal{ref}\cdot \veck_N^\vecb$ and $\vectorsym{\zeta}_N = 2^{-d}\sum_\vecb \widehat{\matF_N^{\vecb,\star}} \cdot \widebar{\veck_N^\vecb}$. The scheme is summarized in \cref{alg:FEM}, we called it ``the FEM scheme'' for it is essentially a FEM with the FFT acting as a preconditioner \cite{Saad2003}.

\begin{algorithm}
\caption{The FEM scheme}
\label{alg:FEM}
\begin{algorithmic}[1]
\STATE{Set $\tenC^\textnormal{ref}$ and $\matE$, initiate a tensor variable $\vecu_N^\star \in l(\I_N;\R^d)$ with $\vecu_N^\star[\vecI]=\vecZero$ for all $\vecI \in \I_N$}
\WHILE{not meet the convergence criterion}
\STATE{Evaluate $\matrixsym{\tau}^{\vecb,\star}_N[\vecI] = \Brackets{\tenC_N^\star[\vecI]-\tenC^\textnormal{ref}}:\Brackets{\nablaS \vecu_N(\vecx_\vecI^\vecb)+\matE}$ for all $\vecI \in \I_N$ and $\vecb \in \Braces{-1,1}^d$, where $\nablaS \vecu_N(\vecx_\vecI^\vecb)$ is obtained by the trilinear interpolation of $\vecu^\star_N$}
\STATE{Perform FFTs $\widehat{\matrixsym{\tau}_N^{\vecb, \star}}=\FFT\Brackets{\matrixsym{\tau}_N^{\vecb,\star}}$ for all $\vecb \in \Braces{-1,1}^d$}
\STATE{Initiate a temporary tensor variable $\widehat{\vecv_N^\star}$ with $\widehat{\vecv_N^\star}[\vecZero]=\vecZero$ and evaluate $\widehat{\vecv_N^\star}[\vecxi]=\matB_N^{-1}[\vecxi] \cdot \vectorsym{\zeta}_N[\vecxi]$ for nonzero $\vecxi$, where $\matB_N=2^{-d}\sum_\vecb \widebar{\veck_N^\vecb} \cdot \tenC^\textnormal{ref}\cdot \veck_N^\vecb$ and $\vectorsym{\zeta}_N = 2^{-d}\sum_\vecb \widehat{\matrixsym{\tau}_N^{\vecb,\star}} \cdot \widebar{\veck_N^\vecb}$}
\STATE{Perform an inverse FFT $\vecv_N^\star=\FFT^{-1}\Brackets{\widehat{\vecv_N^\star}}$}
\STATE{Calculate the convergence indicator via $\vecu_N^\star$ and $\vecv_N^\star$ and $\vecu_N^\star \leftarrow \vecv_N^\star$}
\ENDWHILE
\RETURN $\vecu_N^\star$ and $N^{-d}\sum_{\vecI\in \I_N} \tenC_N^\star[\vecI]:\Brackets{\nablaS\vecu_N(\vecx_\vecI)+\matE}$
\end{algorithmic}
\end{algorithm}

Currently, most researches on FFT-based homogenization are focused on developing fast algorithms to accelerate convergences of iterating processes and designing schemes to stabilize performances on high/infinite contrast materials \cite{Eyre1999,Michel2001,Zeman2010,Eloh2019,Schneider2020}, while few of them consider the convergences of those methods with respect to the spatial resolution $N$, which could be viewed as an analogy of $h$-estimate theories in FEMs \cite{Ciarlet1991,Brenner2008}. To our knowledge, the work by Schneider \cite{Schneider2015} is only published one which seriously discusses such a question. In his article, a priori error estimate of Moulinec-Suquet's scheme is formulated with the trigonometric interpolation operator \cite{Zygmund1968}. However, because of involving pointwise evaluations, the trigonometric interpolation operator is to some certain ``incompatible'' with Lebesgue integrable functions.

Since the main objective in this article is studying convergences with respect to spatial resolutions, it is necessary to clarify the relation between the provided information $\tenC_N^\star$ and the ground truth $\tenC$. As mentioned previously, the ground truth is hidden from the scheme part, we hence cannot hold an optimistic anticipation on convergence rates like immersed FEMs \cite{Li2003,Chen2009} or unfitted FEMs \cite{Burman2014,Huang2017,Chen2021}. The assumptions for $\tenC$ and $\tenC_N^\star$ is elucidated as follows.

\paragraph{Assumption A} The RVE domain $Y$ consists of $(M+1)$ Lipschitz subdomains that are labeled by $D_0,D_1,\dots,D_M$, and $\tenC(\vecx)$ only takes a constant tensor $\tensorsym{\Lambda}_l$ in each subdomain $D_l$.

\paragraph{Assumption B} The constant tensors $\Braces{\tensorsym{\Lambda}_l}_{0\leq l\leq M}$ satisfy properties: (symmetricity) for any $0\leq l\leq M$,
\[
\SquareBrackets{\tensorsym{\Lambda}_l}_{mnpq}=\SquareBrackets{\tensorsym{\Lambda}_l}_{nmpq}=\SquareBrackets{\tensorsym{\Lambda}_l}_{mnqp}=\SquareBrackets{\tensorsym{\Lambda}_l}_{pqmn}, \forall 1\leq m,n,p,q\leq d;
\]
(coercivity) there exist positive constants $\Lambda'$ and $\Lambda''$ such that for any $\matF\in \Sym^d$ and $0\leq l\leq M$,
\[
\Lambda' \matF : \matF \leq \matF : \tensorsym{\Lambda}_l:\matF \leq \Lambda'' \matF : \matF.
\]

\paragraph{Assumption C} For any $\vecI \in \I_N$, if $Y_\vecI \subset D_l$ then $\tenC_N^\star[\vecI]=\tensorsym{\Lambda}_l$, else $\tenC_N^\star[\vecI]$ will be arbitrarily chosen from $\Braces{\tensorsym{\Lambda}_{l'}:Y_\vecI \cap D_{l'} \neq \varnothing}$.

In some cases, we will use the notation $C(p_1, p_2, \dots, p_n)$ to represent a \emph{positive} constant $C$ which depends on $p_1,p_2,\dots, p_n$.

Here are the contributions and structure of this article:

\begin{enumerate}
\item In \cref{sec:MS}, we rebuild the theories in \cite{Schneider2015} on the convergence of Moulinec-Suquet's scheme by introducing a new operator which is will-defined for Lebesgue integrable functions.
\item In \cref{sec:WL}, we prove the convergence of the effective coefficients obtained by Willot's scheme.
\item In \cref{sec:FEM}, by assuming some suitable regularities and combining several priori estimates in FEM theories, we present convergence \emph{rates} of the solution and the effective coefficients derived by the FEM scheme.
\end{enumerate}

\section{Convergence of Moulinec-Suquent's scheme} \label{sec:MS}

Let $S_N$ be a \emph{complex} trigonometric polynomial space
\[
\Braces{f(\vecx)=\sum_{\vecxi\in \mathcal{F}_N}c[\vecxi]\exp(2\pi \iu \vecxi\cdot \vecx): c[\vecxi]\in \mathbb{C}},
\]
and $\Proj_N:L^2(Y;\Cfield)\rightarrow L^2(Y;\Cfield)$ be the orthogonal projection onto $S_N$ \cite{Conway1990}. Take $G_N(\vecx)\coloneqq N^d \mathbb{1}_{C_N}(\vecx)$ where the set $C_N$ is the cube with $\vecZero$ as its center and $1/N$ as its edge length. In the following analysis, for a function $f\in L^2(Y;\Cfield)$, the convolution $G_N*f$ should be understood as applying $G_N$ to the \emph{periodic extension} of $f$, which induces that
\[
G_N*f(\vecx) =\sum_{\vecxi \in \Z^d} \widehat{G_N}[\vecxi]\widehat{f}[\vecxi]\exp(2\pi\iu \vecxi\cdot \vecx),
\]
where
\begin{equation}
\widehat{G_N}[\vecxi]=\left\{
\begin{aligned}
&\prod_{m=1}^{d}\frac{N\sin\Brackets{\frac{\pi \vecxi_m}{N}}}{\pi \vecxi_m}, & \text{if}\ \prod_{m=1}^d\vecxi_m \neq 0, \\
&1, &\text{otherwise}.
\end{aligned}
\right.
\end{equation}

It is easy to see that $(2/\pi)^d \leq \widehat{G_N}[\vecxi] \leq 1$ for $\vecxi \in \mathcal{F}_N$, and we can hence define an operator $\Q_N$ that plays an essential role in our analysis:
\begin{definition}
For any $f\in L^2(Y;\Cfield)$,
\[
\Q_N f(\vecx) = \sum_{\vecxi\in \mathcal{F}_N} \widehat{G_N}^{-1}[\vecxi]\exp\Brackets{-\pi \iu \frac{\vecxi \cdot \vectorsym{1}}{N}} \widehat{f^\star}[\vecxi] \exp(2\pi \iu \vecxi\cdot \vecx),
\]
where $\vectorsym{1}=[1,1,1]$ and
\[
f^\star[\vecI] = G_N*f(\vecx_\vecI) = \fint_{Y_\vecI} f(\vecx) \dx \vecx.
\]
\end{definition}

We can show the following facts of $\Q_N$:
\begin{proposition} \label{prop:QN}
The following statements hold true for $\Q_N$
\begin{itemize}
\item Let $f\in L^2(Y;\Cfield)$, then $\Q_N f\in S_N$, and if $f\in S_N$ then $\Q_N f = f$;
\item Let $f\in L^2(Y;\Cfield)$, then $\Norm{\Q_N f}_{L^2(Y)} \leq \Brackets{\pi/2}^{d}\Norm{f}_{L^2(Y)}$ and
\[
\Norm{f-\Q_N f}_{L^2(Y)} \leq \Braces{1+\Brackets{\pi/2}^d}\Norm{f-\Proj_Nf}_{L^2(Y)};
\]
\item for a series of $\Braces{f_N}\subset L^2(Y;\Cfield)$ with $f_N \rightharpoonup f_\infty$, then $\Q_N f_N \rightharpoonup f_\infty$.
\end{itemize}
\end{proposition}

\begin{proof}
From the definition of $S_N$, it follows that $\Q_N f\in S_N$. Moreover, if $f(\vecx)=\sum_{\vecxi\in \mathcal{F}_N} \widehat{f}[\vecxi]\exp(2\pi\iu \vecxi\cdot \vecx)$, we have
\[
\begin{aligned}
f^\star[\vecI]&=\sum_{\vecxi\in \mathcal{F}_N} \widehat{G_N}[\vecxi] \widehat{f}[\vecxi]\exp(2\pi\iu \vecxi\cdot \vecx_\vecI)\\
&=\sum_{\vecxi\in \mathcal{F}_N} \widehat{G_N}[\vecxi] \widehat{f}[\vecxi] \exp\Brackets{\pi \iu \frac{\vecxi \cdot \vectorsym{1}}{N}} \exp\Brackets{2\pi \iu \frac{\vecxi\cdot \vecI}{N}} \\
&=\sum_{\vecxi\in \mathcal{F}_N} \widehat{f^\star}[\vecxi] \exp\Brackets{2\pi \iu \frac{\vecxi\cdot \vecI}{N}}.
\end{aligned}
\]
Then $\Q_N f=f$ holds by taking $\widehat{f^\star}[\vecxi]=\widehat{G_N}[\vecxi] \widehat{f}[\vecxi] \exp\Brackets{\pi \iu \frac{\vecxi \cdot \vectorsym{1}}{N}}$ into the definition of $\Q_N f$.

From the fact $(2/\pi)^d \leq \widehat{G_N}[\vecxi] \leq 1$ for $\vecxi \in \mathcal{F}_N$, we have
\[
\begin{aligned}
\Norm{\Q_N f}_{L^2(Y)}^2&=\sum_{\vecxi\in \mathcal{F}_N} \Seminorm{\widehat{G_N}^{-1}[\vecxi]}^2 \Seminorm{\widehat{f^\star}[\vecxi]}^2 &\qquad ~  \\
& \leq \Brackets{\frac{\pi}{2}}^{2d} \sum_{\vecxi\in \mathcal{F}_N} \Seminorm{\widehat{f^\star}[\vecxi]}^2 &\qquad ~\\
& = \Brackets{\frac{\pi}{2}}^{2d} \frac{1}{N^d} \sum_{\vecI\in \I_N} \Seminorm{f^\star[\vecI]}^2 &\qquad (\text{Parseval's theorem}) \\
& = \Brackets{\frac{\pi}{2}}^{2d} \frac{1}{N^d} \sum_{\vecI\in \I_N} \Seminorm{\fint_{Y_\vecI} f(\vecx)\dx \vecx}^2 &\qquad (\text{by the definition of }f^\star[\vecI])\\
& \leq  \Brackets{\frac{\pi}{2}}^{2d} \frac{1}{N^d} \sum_{\vecI\in \I_N} \fint_{Y_\vecI} \Seminorm{ f(\vecx)}^2 \dx \vecx  &\qquad (\text{Jensen's inequality}) \\
& = \Brackets{\frac{\pi}{2}}^{2d} \Norm{f}^2_{L^2(Y)}.
\end{aligned}
\]
The estimate of $\Norm{f-\Q_Nf}_{L^2(Y)}$ is a direct corollary via
\[
f-\Q_N f = f-\Proj_Nf+\Proj_Nf-\Q_N f= \Brackets{f-\Proj_Nf} + \Q_N\Brackets{\Proj_Nf-f}.
\]

For checking the weak convergence of $\Braces{\Q_N f_N}$, we only need to show that
\[
\AngleBrackets{\exp(-2\pi\iu\vecxi'\cdot \vecx)\Q_Nf_N}\rightarrow \AngleBrackets{\exp(-2\pi\iu\vecxi'\cdot \vecx)f_\infty}
\]
for any $\vecxi'\in \Z^d$. Take a large enough $N$ such that $\vecxi'\in \mathcal{F}_N$,
\[
\AngleBrackets{\exp(-2\pi\iu\vecxi'\cdot \vecx)\Q_Nf_N}=\widehat{G_N}^{-1}[\vecxi']\exp\Brackets{-\pi \iu \frac{\vecxi'\cdot \vecOne}{N}} \widehat{f_N^\star}[\vecxi'],
\]
and we are left to show $\widehat{f_N^\star}[\vecxi']\rightarrow \widehat{f_\infty}[\vecxi]$. By a direct calculation,
\[
\begin{aligned}
\widehat{f_N^\star}[\vecxi']=&N^{-d}\sum_{\vecI\in \I_N} f^\star_N[\vecI]\exp\Brackets{-2\pi\iu\frac{\vecI\cdot\vecxi'}{N}}=\sum_{\vecI\in \I_N} \int_{Y_\vecI} f_N(\vecx)\dx \vecxi \exp\Brackets{-2\pi\iu\frac{\vecI\cdot\vecxi'}{N}} \\
=&\AngleBrackets{f_N(\vecx)\exp(-2\pi\iu\vecx\cdot\vecxi')} \\
&+ \sum_{\vecI\in \I_N} \int_{Y_\vecI} f_N(\vecx)\Braces{\exp\Brackets{-2\pi\iu\frac{\vecI\cdot\vecxi'}{N}}-\exp(-2\pi\iu\vecx\cdot\vecxi')}\dx \vecx \\
\coloneqq& J_N'+J_N''.
\end{aligned}
\]
We have $J_N'\rightarrow \AngleBrackets{f_\infty\exp(-2\pi\iu\vecx\cdot\vecxi')}$ due to the weak convergence of $\Braces{f_N}$, and it is easy to see that
\[
\Seminorm{J_N''}\leq C(d)\frac{\Seminorm{\vecxi'}}{N} \Norm{f_N}_{L^2(Y)}\leq C(d)\frac{\Seminorm{\vecxi'}}{N} \sup_{N'}\Norm{f_{N'}}_{L^2(Y)},
\]
which implies $J_N''\rightarrow 0$ as $N\rightarrow \infty$. We hence complete the proof of $\Q_N f_N\rightharpoonup f_\infty$.
\end{proof}

\begin{remark}
The operator $\Q_N$ defined here is a natural modification of the trigonometric interpolation in \cite{Schneider2015}. Generally, the coefficients $\tenC$ and strain field $\matVarep$ are discontinuous, while the trigonometric interpolation requires pointwise evaluations in the domain $Y$, which is problematic for Lebesgue integrable functions.
\end{remark}

The next lemma reveals the relation between Moulinec-Suquent's scheme (i.e., \cref{alg:MS}) and the original variational problem \cref{eq:varia form}.
\begin{lemma} \label{lem:MS equi}
The following statements are equivalent:
\begin{enumerate}
\item There exists $\matVarep_N^\star$ such that
\begin{equation}\label{eq:MS form1}
\widehat{\matVarep_N^\star}[\vecxi]=\left\{
\begin{aligned}
&-\widehat{\tensorsym{\Gamma}^0}[\vecxi]:\widehat{\matrixsym{\tau}_N^\star}[\vecxi], & \forall \vecxi\neq \vecZero, \\
& \matE, & \vecxi=\vecZero,
\end{aligned}
\right.
\end{equation}
where $\matrixsym{\tau}_N^\star[\vecI]=\Brackets{\tenC_N^\star[\vecI]-\tenC^\textnormal{ref}}:\matVarep_N^\star[\vecI]$ for all $\vecI \in \I_N$.

\item There exists $\vecu_N \in S_N^d$ with $\AngleBrackets{\vecu_N}=\vecZero$, such that for any $\vecv_N \in S_N^d$
\begin{equation}\label{eq:MS form2}
\begin{aligned}
&N^{-d}\sum_{\vecI\in \I_N} \fint_{Y_\vecI}  \nabla^s\widebar{\vecv_N} \dx \vecx : \tenC_N^\star[\vecI] : \fint_{Y_\vecI} \nabla^s \vecu_N \dx \vecx\\
=&-N^{-d}\sum_{\vecI\in \I_N} \fint_{Y_\vecI} \nabla^s\widebar{\vecv_N} \dx \vecx : \tenC_N^\star[\vecI] : \matE.
\end{aligned}
\end{equation}

\item There exists $\vecu_N \in S_N^d$ with $\AngleBrackets{\vecu_N}=\vecZero$, such that for any $\vecv \in H^1_\#(Y;\Cfield^d)$
\begin{equation}\label{eq:MS form3}
\AngleBrackets{\mathcal{R}_N\Brackets{\nablaS\widebar{\vecv}}:\Q_N\Brackets{\tenC_N:\nablaS \vecu_N}}=-\AngleBrackets{\mathcal{R}_N\Brackets{\nablaS\widebar{\vecv}}:\Q_N\Brackets{\tenC_N:\matE}},
\end{equation}
where the operator $\mathcal{R}_N$ is defined as $\mathcal{R}_N f=G_N*G_N*f$.
\end{enumerate}
\end{lemma}
\begin{remark}
Because we have not imposed any regularity assumptions, any statement above may not be true. The purpose of this lemma is showing \cref{eq:MS form1,eq:MS form2,eq:MS form3} are equivalent transformations, and we can jump out the original scheme description \cref{alg:MS} and study a more ``mathematical'' formulation, i.e. \cref{eq:MS form3}.
\end{remark}
\begin{proof}
We first prove that statements $1$ and $2$ are equivalent. Take an expression of $\vecv_N$ as $\vecv_N(\vecx)=\sum_{\vecxi\in \mathcal{F}_N}\widehat{\vecv_N}[\vecxi]\exp(2\pi\iu \vecxi\cdot\vecx)$, we have
\[
\fint_{Y_\vecI} \nablaS \vecv_N \dx \vecx=\sum_{\vecxi\in \mathcal{F}_N} 2\pi \iu \vecxi \otimesS \widehat{\vecv_N^\star}[\vecxi]\exp\Brackets{2\pi\iu\frac{\vecxi\cdot \vecI}{N}},
\]
where $\widehat{\vecv_N^\star}$ is derived from the DFT of $\vecv_N^\star[\vecI]=\fint_{Y_\vecI}\vecv_N\dx \vecx$. By splitting $\tenC_N^\star=\tenC^\textnormal{ref}+\delta\tenC_N^\star$ and applying Parseval's theorem, we can convert \cref{eq:MS form2} into
\[
\begin{aligned}
&N^{-d}\sum_{\vecI\in \I_N} \fint_{Y_\vecI} \nabla^s \widebar{\vecv_N} \dx \vecx : \tenC^\textnormal{ref} : \fint_{Y_\vecI} \nabla^s \vecu_N \dx \vecx\\
=&\sum_{\vecxi\in \mathcal{F}_N} 4\pi^2 \Brackets{\vecxi \otimesS \widebar{\widehat{\vecv_N^\star}}}:\tenC^\textnormal{ref}:\Brackets{\vecxi \otimesS \widehat{\vecu_N^\star}} \\
=&-N^{-d}\sum_{\vecI\in \I_N} \fint_{Y_\vecI} \nabla^s \widebar{\vecv_N} \dx \vecx : \Braces{\delta\tenC^\star_N : \fint_{Y_\vecI} \nabla^s \vecu_N \dx \vecx+\tenC_N^\star:\matE} \\
=&\sum_{\vecxi\in \mathcal{F}_N} 2\pi \iu \Brackets{\vecxi \otimesS \widebar{\widehat{\vecv_N^\star}}}: \widehat{\delta \tenC_N^\star:\mattau_N^\star},
\end{aligned}
\]
where $\mattau_N^\star[\vecI]=\delta\tenC^\star_N : \Brackets{\fint_{Y_\vecI} \nabla^s \vecu_N \dx \vecx+\matE}$. Note this exactly repeats the derivation of the Lipmann-Schwinger equation \cref{eq:LS}, we hence show that the statement $2$ leads to $1$. Meanwhile, according the definition of $\widehat{\tenGamma^0}$, we can find $\widehat{\vecu_N^\star}$ such that $\widehat{\matVarep_N^\star}[\vecxi]=2\pi\iu\vecxi \otimesS \widehat{\vecu_N^\star}[\vecxi]$ for $\vecxi \neq \vecZero$. Then, via reverse steps in proving ($2\Rightarrow 1$), we complete the proof of equivalence between statements $1$ and $2$.

We then prove that statements $2$ and $3$ are equivalent. Without loss of generality, we assume $\vecv \in S_N^d$ in \cref{eq:MS form3} and $\vecv(\vecx)=\sum_{\vecxi\in \mathcal{F}_N} \widehat{\vecv}[\vecxi]\exp(2\pi\iu \vecxi\cdot \vecx)$. Then recalling the definitions of $\mathcal{R}_N$ and $\mathcal{Q}_N$, we have
\[
\begin{aligned}
&\AngleBrackets{\nablaS\Brackets{\mathcal{R}_N \widebar{\vecv}}:\Q_N\Brackets{\tenC_N:\nablaS \vecu_N}}\\
=&\sum_{\vecxi\in \mathcal{F}_N} -2\pi\iu \Brackets{\widehat{G_N}}^2\Brackets{\vecxi \otimesS\widebar{\widehat{\vecv}}[\vecxi]}:\Brackets{\widehat{G_N}}^{-1}\exp\Brackets{-\pi\iu \frac{\vecxi\cdot \vecOne}{N}}\widehat{\tenC_N^\star:\matU^\star}[\vecxi] \\
=&\sum_{\vecxi\in \mathcal{F}_N} \widebar{\widehat{\matV^\star}}[\vecxi]:\widehat{\tenC_N^\star:\matU^\star}[\vecxi] \\
=&N^{-d} \sum_{\vecI\in \I_N} \widebar{\matV^\star}[\vecI]:\tenC_N^\star[\vecI]:\matU^\star[\vecI],
\end{aligned}
\]
where $\matV^\star[\vecI]=\fint_{Y_\vecI} \nablaS \vecv \dx \vecx$ and $\matU^\star[\vecI]=\fint_{Y_\vecI} \nablaS \vecu_N \dx \vecx$. Similarly,
\[
-\AngleBrackets{\nablaS\Brackets{\mathcal{R}_N \widebar{\vecv}}:\Q_N\Brackets{\tenC_N:\matE}}=-N^{-d} \sum_{\vecI\in \I_N} \widebar{\matV^\star}[\vecI]:\tenC_N^\star[\vecI]:\matE.
\]
We hence arrive at ($2\Leftrightarrow 3$).
\end{proof}

The following theorem provides a priori estimate and a convergence proof of the basic scheme.
\begin{theorem}\label{thm:MS main}
Let $\tenC$ and $\tenC_N$ satisfy Assumptions A-C, $\vecu$ be the solution of \cref{eq:varia form}. Then there exists a unique solution $\vecu_N$ of the variation problem \cref{eq:MS form3} with an estimate
\[
\begin{aligned}
&\Lambda' \Norm{\nablaS \Brackets{\vecu_N-\vecu}}_{L^2(Y)}\\
\leq & C\Braces{\Lambda''\Norm{\matVarep-\Proj_N \matVarep}_{L^2(Y)}+\Norm{\matsigma-\Proj_N\matsigma}_{L^2(Y)}+ \Norm{\Brackets{\tenC-\tenC_N}:\matVarep}_{L^2(Y)}}
\end{aligned}
\]
where $\matVarep=\nablaS \vecu+\matE$, $\matsigma=\tenC:\matVarep$ and $C$ is a positive constant.
\end{theorem}
\begin{proof}
We need first prove a coercivity estimate. Taking any $\vecv_N(\vecx) \in S_N^d$ with $\widehat{\vecv_N}[\vecxi]$ as its Fourier coefficients, recalling the proof of \cref{lem:MS equi}, we have
\[
\begin{aligned}
&\AngleBrackets{\nablaS\Brackets{\mathcal{R}_N \widebar{\vecv_N}}:\Q_N\Brackets{\tenC_N:\nablaS \vecv_N}} \\
=&N^{-d}\sum_{\vecI\in \I_N}\fint_{Y_\vecI} \nablaS \widebar{\vecv_N} \dx \vecx :\tenC_N^\star[\vecI]: \fint_{Y_\vecI} \nablaS \vecv_N \dx \vecx \\
\geq& N^{-d}\Lambda'\sum_{\vecI\in \I_N}\fint_{Y_\vecI} \nablaS \widebar{\vecv_N} \dx \vecx : \fint_{Y_\vecI} \nablaS \vecv_N \dx \vecx \\
=& \Lambda' \sum_{\vecxi\in \mathcal{F}_N} \Seminorm{\widehat{\matV^\star}}^2[\vecxi],
\end{aligned}
\]
where $\matV^\star[\vecI]=\fint_{Y_\vecI} \nablaS \vecv_N\dx \vecx$. Note that for any $\vecxi\in \mathcal{F}_N$,
\[
\Seminorm{\widehat{\matV^\star}}[\vecxi]=\widehat{G_N}\Seminorm{2\pi\iu\vecxi\otimesS \widehat{\vecv_N}}\geq \Brackets{\frac{2}{\pi}}^d \Seminorm{\widehat{\nablaS \vecv_N}}[\vecxi],
\]
which leads
\[
\AngleBrackets{\nablaS\Brackets{\mathcal{R}_N \widebar{\vecv_N}}:\Q_N\Brackets{\tenC_N:\nablaS \vecv_N}} \geq \Lambda' \Brackets{\frac{2}{\pi}}^{2d} \Norm{\nablaS \vecv_N}_{L^2(Y)}^2.
\]

Replace $\vecv_N$ with $e_N=\vecu_N-\Proj_N \vecu$ in the above inequality and combine the variational equalities \cref{eq:varia form,eq:MS form3},
\[
\begin{aligned}
&\Lambda' \Brackets{\frac{2}{\pi}}^{2d} \Norm{\nablaS e_N}_{L^2(Y)}^2 \\
\leq& -\AngleBrackets{\nablaS\Brackets{\mathcal{R}_N \widebar{e_N}}:\Q_N\Brackets{\tenC_N:\matE}}-\AngleBrackets{\nablaS\Brackets{\mathcal{R}_N \widebar{e_N}}:\Q_N\Brackets{\tenC_N:\nablaS \Proj_N \vecu}} \\
=& \AngleBrackets{\nablaS\Brackets{\mathcal{R}_N \widebar{e_N}}:\Brackets{\tenC:\nablaS \vecu + \tenC:\matE -\Q_N\Brackets{\tenC_N:\matE+ \tenC_N:\nablaS \Proj_N \vecu}}} \\
=& \AngleBrackets{\nablaS\Brackets{\mathcal{R}_N \widebar{e_N}}:\Brackets{\matsigma-\Q_N\matsigma -\Q_N\Brackets{\tenC_N:\Brackets{\Proj_N\matVarep-\matVarep}+\Brackets{\tenC_N-\tenC}:\matVarep}}}.
\end{aligned}
\]
It is easy to show that $\Norm{\nablaS\Brackets{\mathcal{R}_N \widebar{e_N}}}_{L^2(Y)}=\Norm{\mathcal{R}_N\Brackets{\nablaS \widebar{e_N}}}_{L^2(Y)}\leq \Norm{\nablaS e_N}_{L^2(Y)}$. Finally, the priori estimate follows from \cref{prop:QN} and H\"{o}lder's inequality.
\end{proof}

Denote by $\tenC_N^{\textnormal{eff},\textnormal{B}}$ the effective coefficients obtained by Moulinec-Suquent's scheme, based on the equivalent results in \cref{lem:MS equi}, we have
\[
\begin{aligned}
\tenC_N^{\textnormal{eff},\textnormal{B}}:\matE=&N^{-d}\sum_{\vecI\in \I_N} \tenC_N^\star[\vecI]:\matVarep_N^\star[\vecI]\\
=& N^{-d}\sum_{\vecI\in \I_N}\tenC_N^\star[\vecI]:\Brackets{\fint_{Y_\vecI} \nablaS\vecu_N\dx \vecx+\matE}\\
=&\AngleBrackets{\tenC_N:\Brackets{\nablaS \vecu_N+\matE}}.\\
\end{aligned}
\]
The following theorem shows the convergence of the effective coefficients.
\begin{theorem}
Under Assumptions A-C, $\tenC_N^{\textnormal{eff},\textnormal{B}}$ converges to $\tenC^\textnormal{eff}$ as $N\rightarrow \infty$, where $\tenC^\textnormal{eff}$ is defined by \cref{eq:eff}.
\end{theorem}

\section{Convergence of Willot's scheme} \label{sec:WL}

Let $S_{N\shortminus} \subset S_N$ be defined as
\[
S_{N\shortminus}\coloneqq \Braces{f(\vecx)=\sum_{\vecxi\in \mathcal{F}_{N\shortminus}}c[\vecxi]\exp(2\pi \iu \vecxi\cdot \vecx): c[\vecxi]\in \mathbb{C}},
\]
denote by $V_{N\shortminus}$ a subset of $V_N$ with
\[
V_{N\shortminus}\coloneqq\Braces{v_N\in V_N: v_N\Brackets{\frac{\vecI}{N}}=\sum_{\vecxi\in \mathcal{F}_{N\shortminus}}c[\vecxi]\exp\Brackets{2\pi\iu\frac{\vecxi\cdot \vecI}{N}},\forall \vecI\in \I_N },
\]
we first present a similar lemma of \cref{lem:MS equi} for Willot's scheme.
\begin{lemma}\label{lem:WL equi}
The following statements are equivalent:
\begin{enumerate}
\item There exists $\matVarep_N^\star$ such that
\begin{equation}\label{eq:WL form1}
\widehat{\matVarep_N^\star}[\vecxi]=\left\{
\begin{aligned}
&-\widehat{\tensorsym{\Gamma}^0_\textnormal{W}}[\vecxi]:\widehat{\matrixsym{\tau}_N^\star}[\vecxi], & \forall \vecxi\in \mathcal{F}_{N\shortminus}\setminus\Braces{\vecZero}, \\
& \matE, & \vecxi=\vecZero, \\
& \vecZero, &\forall \vecxi\in \mathcal{F}_N\setminus \mathcal{F}_{N\shortminus}
\end{aligned}
\right.
\end{equation}
where $\matrixsym{\tau}_N^\star[\vecI]=\Brackets{\tenC_N^\star[\vecI]-\tenC^\textnormal{ref}}:\matVarep_N^\star[\vecI]$ for all $\vecI \in \I_N$.

\item There exists $\vecu_N \in V_{N\shortminus}^d$ with $\AngleBrackets{\vecu_N}=\vecZero$, such that for any $\vecv_N \in V_{N\shortminus}^d$
\begin{equation}\label{eq:WL form2}
\begin{aligned}
&N^{-d}\sum_{\vecI\in \I_N}  \nabla^s\vecv_N(\vecx_\vecI) : \tenC_N^\star[\vecI] : \nabla^s \vecu_N(\vecx_\vecI)\\
=&-N^{-d}\sum_{\vecI\in \I_N} \nabla^s\vecv_N(\vecx_\vecI) : \tenC_N^\star[\vecI] : \matE.
\end{aligned}
\end{equation}

\item There exists $\vecu_N \in S_{N\shortminus}^d$ with $\AngleBrackets{\vecu_N}=\vecZero$, such that for any $\vecv_N \in S_{N\shortminus}^d$
\begin{equation}\label{eq:WL form3}
\begin{aligned}
&N^{-d}\sum_{\vecI\in \I_N} \fint_{Y_\vecI}  \DopS_N \widebar{ \vecv_N} \dx \vecx : \tenC_N^\star[\vecI] : \fint_{Y_\vecI} \DopS_N \vecu_N \dx \vecx\\
=&-N^{-d}\sum_{\vecI\in \I_N} \fint_{Y_\vecI} \DopS_N \widebar{\vecv_N} \dx \vecx : \tenC_N^\star[\vecI] : \matE,
\end{aligned}
\end{equation}
where the operator $\DopS_N$ is defined as
\[
\DopS_N f=\sum_{\vecxi\in \mathcal{F}_{N\shortminus}} \veck_N[\vecxi] \otimesS \widehat{f}[\vecxi] \exp(2\pi\iu \vecxi\cdot\vecx)
\]
for $f\in L^2(Y;\Cfield)$.

\item There exists $\vecu_N \in S_{N\shortminus}^d$ with $\AngleBrackets{\vecu_N}=\vecZero$, such that for any $\vecv \in H^1_\#(Y;\Cfield^d)$
\begin{equation}\label{eq:WL form4}
\AngleBrackets{\mathcal{R}_N\Brackets{\DopS_N \widebar{ \vecv}}:\Q_N\Brackets{\tenC_N:\DopS_N \vecu_N}}=-\AngleBrackets{\mathcal{R}_N\Brackets{\DopS_N \widebar{ \vecv}}:\Q_N\Brackets{\tenC_N:\matE}}.
\end{equation}
\end{enumerate}
\end{lemma}

The proof of this lemma is exactly same as \cref{lem:MS equi}. The variational form \cref{eq:WL form2} could be treated as a reduced integration FEM. Note that \cref{eq:WL form3,eq:WL form4} are parallel versions to \cref{eq:MS form2,eq:MS form3}, while original $\nablaS$ is replaced with $\DopS_N$. The difficulty in analyzing Willot's scheme is losing coercivity, and we will explain it as follows. It is easy to see that $\DopS_N \widebar{f}=\widebar{\DopS_N f}$, and the left-hand side of \cref{eq:WL form3} gives
\[
\begin{aligned}
&N^{-d}\sum_{\vecI\in \I_N} \fint_{Y_\vecI}  \DopS_N \widebar{ \vecv_N} \dx \vecx : \tenC_N^\star[\vecI] : \fint_{Y_\vecI} \DopS_N \vecv_N \dx \vecx \\
\geq& \Lambda'N^{-d} \sum_{\vecI\in \I_N} \Seminorm{\fint_{Y_\vecI} \DopS_N\vecv_N\dx\vecx}^2 \\
=&\Lambda' \sum_{\vecxi\in \mathcal{F}_{N\shortminus}} \Seminorm{\veck_N\otimesS \widehat{\vecv^\star_N}}^2[\vecxi] \\
\geq&\frac{\Lambda'}{2} \sum_{\vecxi\in \mathcal{F}_{N\shortminus}} \Seminorm{\veck_N}^2 \Seminorm{\widehat{\vecv_N^\star}}^2.
\end{aligned}
\]
where $\vecv_N^\star[\vecI]=\fint_{Y_\vecI} \vecv_N\dx \vecx$. Then recall the definition of $\veck_N$ \cref{eq:WL k_N},
\[
\begin{aligned}
\Seminorm{\veck_N}^2&=\frac{1}{64}\Braces{\prod_{m=1}^d\Seminorm{\exp\Brackets{2\pi\iu\frac{\vecxi_m}{N}}+1}^2}\Braces{\sum_{n=1}^d \Seminorm{2N\tan\Brackets{\pi\frac{\vecxi_n}{N}}}^2}\\
& =\Braces{\prod_{m=1}^d \cos^2\Brackets{\pi\frac{\vecxi_m}{N}}}\Braces{\sum_{n=1}^d \Seminorm{2N\tan\Brackets{\pi\frac{\vecxi_n}{N}}}^2}.
\end{aligned}
\]
According to the basic inequality $\Seminorm{\tan(x)}\geq \Seminorm{x}$ for $x\in (-\pi/2,\pi/2)$, we obtain $\Braces{\sum_{n=1}^d \Seminorm{2N\tan\Brackets{\pi\frac{\vecxi_n}{N}}}^2}\geq 4\pi^2 \Seminorm{\vecxi}^2$, and the coercivity
\[
c\Norm{\nabla \vecv_N}_{L^2(Y)}^2\leq N^{-d}\sum_{\vecI\in \I_N} \fint_{Y_\vecI}  \DopS_N \widebar{ \vecv_N} \dx \vecx : \tenC_N^\star[\vecI] : \fint_{Y_\vecI} \DopS_N \vecv_N \dx \vecx
\]
would emerge if $\prod_{m=1}^d \cos^2\Brackets{\pi\frac{\vecxi_m}{N}}$ can be uniformly bounded below. However, this is not true, because $\min_{\vecxi\in\mathcal{F}_{N\shortminus}}\prod_{m=1}^d \cos^2\Brackets{\pi\frac{\vecxi_m}{N}}=O(N^{-2d})$.

A remedy is considering a modified gradient operator $\TDop_N$ as
\begin{equation}
\TDop_N f(\vecx)\coloneqq \sum_{\vecxi\in \mathcal{F}_{N\shortminus}} \tilde{\veck}_N[\vecxi] \widehat{f}[\vecxi]\exp(2\pi\iu\vecxi\cdot \vecx)
\end{equation}
for $f\in L^2(Y;\Cfield)$, where
\[
\tilde{\veck}_N=2\iu N\SquareBrackets{\tan\Brackets{\pi\frac{\vecxi_1}{N}}, \tan\Brackets{\pi\frac{\vecxi_1}{N}}, \tan\Brackets{\pi\frac{\vecxi_1}{N}}},
\]
and $\TDopS_N$ is the symmetrization of $\TDop_N$ as $\nablaS$ of $\nabla$ for vector-valued functions. We also rewrite \cref{eq:WL form4} as follows: find $\vecw_N\in S_{N\shortminus}^d$ such that for any $\vecv\in H^1_\#(Y;\Cfield^d)$,
\begin{equation}\label{eq:WL form5}
\AngleBrackets{\mathcal{R}_N\Brackets{\TDopS_N \widebar{ \vecv}}:\Q_N\Brackets{\tenC_N:\TDopS_N \vecw_N}}=-\AngleBrackets{\mathcal{R}_N\Brackets{\TDopS_N \widebar{ \vecv}}:\Q_N\Brackets{\tenC_N:\matE}}.
\end{equation}
The motivation behind is an observation that $\vecw_N$ and $\vecu_N$ of \cref{eq:WL form4} are connected with
\[
\frac{1}{8}\prod_{m=1}^d\Brackets{\exp\Brackets{\pi\frac{\vecxi_m}{N}}+1}\widehat{\vecu_N}[\vecxi]= \widehat{\vecw_N}[\vecxi]
\]
for all $\vecxi \in \mathcal{F}_{N\shortminus}$.

To unveil the limiting behavior of $\vecw_N$, we first prove several properties corresponding to $\TDop_N$:
\begin{proposition} \label{prop:TDN}
For any $f\in C^\infty_\#(Y;\Cfield)$, it holds that $\TDop_N f\rightarrow \nabla f$ in $L^2(Y;\Cfield^d)$. If a series $\Braces{f_N} \subset H^1_\#(Y;\Cfield)$ satisfies $f_N \rightharpoonup f_\infty$ in $H^1_\#(Y;\Cfield)$ and $\Braces{\TDop_N f_N} \subset L^2(Y;\Cfield^d)$, then $\TDop_N f_N \rightharpoonup \nabla f_\infty$ in $L^2(Y;\Cfield^d)$.
\end{proposition}
\begin{proof}
For the first argument, after taking the Fourier expansion of $f$, we have
\[
\begin{aligned}
\Norm{\TDop_N f-\nabla f}_{L^2(Y)}^2 &= \sum_{\vecxi\in \mathcal{F}_{N\shortminus}} \Seminorm{\tilde{\veck}_N[\vecxi]-2\pi\iu\vecxi}^2 \Seminorm{\widehat{f}}^2[\vecxi]+ \sum_{\vecxi\in \Z^d\setminus \mathcal{F}_{N\shortminus}} 4\pi \Seminorm{\vecxi}^2 \Seminorm{\widehat{f}}^2[\vecxi] \\
&\coloneqq J_N'+ J_N''.
\end{aligned}
\]
According to smoothness of $f\in C^\infty_\#(Y;\Cfield)$, we have $J_N''\rightarrow 0$ as $N\rightarrow \infty$. We rewrite $J_N'$ as $\sum_{\vecxi \in \Z^d} a_N[\vecxi] \Seminorm{\widehat{f}}^2[\vecxi]$, where
\[
a_N[\vecxi]\coloneqq\left\{
\begin{aligned}
&4\pi^2\sum_{d=1}^m \Seminorm{\frac{N}{\pi}\tan\Brackets{\vecxi_m\frac{\pi}{N}}-\vecxi_m}^2, & N\geq 2\max_m\Seminorm{\vecxi_m}+1 \\
&0, & \text{ else},
\end{aligned}
\right.
\]
By the fact that $\tan(x)/x$ is monotone in $x\in (0,\pi/2)$, we have
\[
\Seminorm{\frac{N}{\pi}\tan\Brackets{\vecxi_m\frac{\pi}{N}}} \leq \frac{2\Seminorm{\vecxi_m}+1}{\pi}\tan\Brackets{\frac{\pi\Seminorm{\vecxi_m}}{2\Seminorm{\vecxi_m}+1}} \leq C\Brackets{2\Seminorm{\vecxi_m}+1}^2
\]
for $N\geq 2\max_m\Seminorm{\vecxi_m}+1$, which implies that there exists a positive constant $C$ such that $\Seminorm{a_N[\vecxi]}\leq C(\Seminorm{\vecxi}^4+1)$. Recalling smoothness of $f$, we arrive at $\sum_{\vecxi\in \Z^d} (1+\Seminorm{\vecxi}^4) \Seminorm{\widehat{f}}^2[\vecxi] < \infty$. It is easy to see that $a_N[\vecxi]\rightarrow 0$ for any $\vecxi$, then $J_N'\rightarrow 0$ follows from the dominated convergence theorem.

For the second argument, it suffices to prove that
\[
\AngleBrackets{\exp(-2\pi\iu\vecxi'\cdot\vecx) \widebar{\vecups} \cdot \TDop_Nf_N} \rightarrow \AngleBrackets{\exp(-2\pi\iu\vecxi'\cdot\vecx) \widebar{\vecups} \cdot \nabla f_\infty}=2\pi\iu \widebar{\vecups}\cdot \vecxi'\widehat{f_\infty}[\vecxi']
\]
for any $\vecxi'\in \Z^d$ and $\vecups \in \Cfield^d$. According to the definition of $\TDop_N$, we get
\[
\AngleBrackets{\exp(-2\pi\iu\vecxi'\cdot\vecx) \widebar{\vecups} \cdot \TDop_Nf_N}=\widebar{\vecups}\cdot\tilde{\veck}_N[\vecxi']\widehat{f}_N[\vecxi']
\]
for large enough $N$. Combining $\tilde{\veck}_N[\vecxi']\rightarrow 2\pi\iu \vecxi$ and $f_N\rightharpoonup f_\infty$ which leads $\widehat{f}_N[\vecxi']\rightarrow \widehat{f_\infty}[\vecxi']$, we hence finish the proof.
\end{proof}

The following lemma says that $\vecw_N$ \emph{weakly} converges to $\vecu$ as $N\rightarrow \infty$.
\begin{lemma} \label{lem:WL key}
Under Assumptions A-C, let $\vecu$ be the solution of \cref{eq:varia form}. Then there exists a unique $\vecw_N$ satisfying \cref{eq:WL form5} such that $\vecw_N\rightharpoonup \vecu$ in $H^1_\#(Y;\Cfield)$ as $N\rightarrow \infty$.
\end{lemma}
\begin{proof}
This proof consists of several steps.

\paragraph{Step1} As previously stated, we have
\[
\begin{aligned}
&\AngleBrackets{\mathcal{R}_N\Brackets{\TDopS_N \widebar{ \vecv_N}}:\Q_N\Brackets{\tenC_N:\TDopS_N \vecv_N}} \geq \frac{\Lambda'}{2} \sum_{\vecxi\in \mathcal{F}_{N\shortminus}} \Seminorm{\tilde{\veck}_N}^2\Seminorm{\widehat{\vecv_N^\star}}^2 \\
\geq & \Brackets{\frac{2}{\pi}}^{2d} \frac{\Lambda'}{2} \sum_{\vecxi\in \mathcal{F}_{N\shortminus}} \Seminorm{\tilde{\veck}_N}^2\Seminorm{\widehat{\vecv_N}}^2 \\
= & \Brackets{\frac{2}{\pi}}^{2d} \frac{\Lambda'}{2} \Norm{\TDop_N \vecv_N}_{L^2(Y)}^2.
\end{aligned}
\]
Meanwhile, for the right-hand side,
\[
\begin{aligned}
&\Seminorm{\AngleBrackets{\mathcal{R}_N\Brackets{\TDopS_N \widebar{ \vecv_N}}:\Q_N\Brackets{\tenC_N:\matE}}}\leq \Norm{\mathcal{R}_N\Brackets{\TDopS_N \vecv_N}}_{L^2(Y)} \Norm{\Q_N\Brackets{\tenC_N:\matE}}_{L^2(Y)} \\
\leq & C\Lambda''\Norm{\TDop_N \vecv_N}_{L^2(Y)} \sqrt{\matE:\matE}.
\end{aligned}
\]
We hence obtain that $\Norm{\TDop_N \vecw_N}_{L^2(Y)}\leq C\Lambda''/\Lambda' \sqrt{\matE:\matE}$. Recalling that $\Seminorm{\tilde{\veck}_N}\geq 2\pi\Seminorm{\vecxi}$, we show that $\Norm{\vecw_N}_{H^1(Y)}$ is uniformly bounded. Up to a subsequence, we derive that $\vecw_N \rightharpoonup \vecw_\infty$ in $H^1_\#(Y;\Cfield^d)$ (such a subsequence is still denoted by $\Braces{\vecw_N}$).

\paragraph{Step2} The goal is proving $\Q_N\Brackets{\tenC_N:\TDopS_N \vecw_N} \rightharpoonup \tenC:\nablaS \vecw_\infty$. According to the results presented in \cref{prop:QN}, we are left to prove that $\tenC_N:\TDopS_N\vecw_N \rightharpoonup \tenC:\nablaS \vecw_\infty$ in $L^2(Y;\Sym^d)$. Arbitrarily choosing $\matV\in L^2(Y;\Sym^d)$, we have
\[
\AngleBrackets{\matV:\tenC_N:\TDopS_N\vecw_N}=\AngleBrackets{\Brackets{\matV:\tenC_N}:\Brackets{\TDopS_N\vecw_N}}.
\]
Note that $\matV:\tenC_N$ converges to $\matV:\tenC$ strongly by Assumption C and $\TDopS_N\vecw_N \rightharpoonup \nablaS \vecw_\infty$ by \cref{prop:TDN}, it comes that
\[
\AngleBrackets{\matV:\tenC_N:\TDopS_N\vecw_N} \rightarrow \AngleBrackets{\matV:\tenC:\nablaS \vecw_\infty},
\]
and we hence show that $\Q_N\Brackets{\tenC_N:\TDopS_N \vecw_N} \rightharpoonup \tenC:\nablaS \vecw_\infty$.

\paragraph{Step3} The final step is to identify $\vecw_\infty$ as $\vecu$. For any $\vecv \in C^\infty_\#(Y;\Cfield^d)$, we have that $\TDopS_N \vecv$ converges to $\nablaS \vecv$ strongly in $L^2(Y;\Sym^d)$ according to \cref{prop:TDN}. It is not difficult to show that $\mathcal{R}_N\Brackets{\TDopS_N \vecv}\rightarrow \nablaS \vecv$ because $\mathcal{R}_N$ is essentially a mollifier. Applying the ``strong-weak'' argument in both sides of \cref{eq:WL form5}, we arrive at
\[
\AngleBrackets{\nablaS \widebar{\vecv}:\tenC:\nablaS\vecw_\infty}=-\AngleBrackets{\nablaS\widebar{\vecv}:\tenC:\matE}
\]
for any smooth $\vecv$, which implies $\vecw_\infty=\vecu$. Moreover, the weak limiting point of any subsequence of $\Braces{\vecw_N}$ is unique, we complete the proof of $\vecw_N \rightharpoonup \vecu$.
\end{proof}

Let $\tenC_N^{\textnormal{eff},W}$ be the effective coefficients obtained by Willot's scheme. It is easy to show that
\[
\begin{aligned}
\tenC_N^{\textnormal{eff},\textnormal{W}}:\matE=&N^{-d}\sum_{\vecI\in \I_N} \tenC_N^\star[\vecI]:\matVarep_N^\star[\vecI]\\
=& N^{-d}\sum_{\vecI\in \I_N}\tenC_N^\star[\vecI]:\Brackets{\fint_{Y_\vecI} \DopS_N\vecu_N\dx \vecx+\matE}\\
=&\AngleBrackets{\tenC_N:\Brackets{\DopS_N \vecu_N+\matE}}\\
=&\AngleBrackets{\tenC_N:\Brackets{\TDopS_N \vecw_N+\matE}},
\end{aligned}
\]
where $\matVarep^\star_N$ is from \cref{eq:WL form1}, $\vecu_N$ is from \cref{eq:WL form3,eq:WL form4}, and $\vecw_N$ is from \cref{eq:WL form5}. A direct result from \cref{lem:WL key} is the convergence of the effective coefficients.
\begin{theorem}
Under Assumptions A-C, $\tenC_N^{\textnormal{eff},\textnormal{W}}$ converges to $\tenC^\textnormal{eff}$ as $N\rightarrow \infty$, where $\tenC^\textnormal{eff}$ is defined by \cref{eq:eff}.
\end{theorem}

\section{Convergence rates of the FEM scheme} \label{sec:FEM}
We emphasize that \cref{alg:FEM} is actually equivalent to solving a variational problem: find $\vecu_N \in V_N^d$ with $\AngleBrackets{\vecu_N}=\vecZero$ such that for any $\vecv_N \in V_N^d$,
\begin{equation}\label{eq:FEM varia}
\AngleBrackets{\nablaS \vecv_N:\tenC_N:\nablaS \vecu_N}=-\AngleBrackets{\nablaS \vecv_N:\tenC_N}:\matE.
\end{equation}

Because $\tenC_N$ is only piecewisely constant and globally discontinuous, we cannot assume a better regularity than $W^{1,\infty}$ of $\vecu$. The local regularization operator proposed in \cite{Bernardi1998} is an appropriate framework for our problem.
\begin{lemma}[cf. \cite{Bernardi1998}]
There exists an operator $\Lop_N:H^1_\#(Y) \rightarrow V_N$ such that for any $v\in H^1_\#(Y)$, $\Norm{v-\Lop_N v}_{H^1(Y)}\rightarrow 0$ as $N\rightarrow \infty$. Moreover, there exists a positive constant $C_\textnormal{int}$, such that if $v\in W^{1,p}(\Delta_\vecI)$ with $1\leq p \leq \infty$,
\[
\Norm{v-\Lop_N v}_{W^{1,p}(Y_\vecI)} \leq C_\textnormal{int} \Norm{v}_{W^{1,p}(\Delta_\vecI)},
\]
and if $v\in H^2(\Delta_\vecI)$,
\[
\Norm{v-\Lop_N v}_{H^1(Y_\vecI)} \leq C_\textnormal{int}  N^{-1}\Norm{v}_{H^2(\Delta_\vecI)},
\]
where $Y_\vecI \subset \Delta_\vecI$ and
\[
\Delta_\vecI=\Brackets{\frac{\vecI_1-1}{N}, \frac{\vecI_1+1}{N}}\times\Brackets{\frac{\vecI_2-1}{N}, \frac{\vecI_2+1}{N}}\times\Brackets{\frac{\vecI_3-1}{N}, \frac{\vecI_3+1}{N}}.
\]
\end{lemma}

Utilizing C\'{e}a's lemma, we could obtain a similar priori estimate as \cref{thm:MS main}
\begin{lemma} \label{lem:FEM pri}
Let $\tenC$ and $\tenC_N$ satisfy Assumptions A-C, $\vecu$ be the solution of \cref{eq:varia form}. Then there exists a unique solution $\vecu_N$ of the variation problem \cref{eq:FEM varia} with an estimate
\[
\Lambda' \Norm{\nablaS \Brackets{\vecu_N-\vecu}}_{L^2(Y)} \leq C\Braces{\Lambda''\Norm{\nablaS\Brackets{\vecu-\Lop_N \vecu}}_{L^2(Y)}+ \Norm{\Brackets{\tenC-\tenC_N}:\matVarep}_{L^2(Y)}}
\]
where $\matVarep=\nablaS \vecu+\matE$ and $C$ is a positive constant.
\end{lemma}

The proof of the above lemma is straight, and we omit it here. Interestingly, the assumption $\Lambda'>0$ could be relaxed as $\Lambda'\geq 0$ in the FEM scheme. Because \cref{eq:FEM varia} minimizes the energy $\AngleBrackets{\nablaS \vecv_N:\tenC_N:\nablaS \vecv_N}$, and the subdomain $\Braces{\tenC_N(\vecx)=\vecZero}$ corresponding to the porous part of the RVE does \emph{not} affect the global coercivity. This good property does not hold for Moulinec-Suquent's scheme, which could be explained from the proof of \cref{thm:MS main}: the coercivity of $\AngleBrackets{\nablaS\Brackets{\mathcal{R}_N \widebar{\vecv_N}}:\Q_N\Brackets{\tenC_N:\nablaS \vecv_N}}$ is proven by converting it into the Fourier space, and such an operation is not applicable under $\Lambda'=0$. This observation to some extent reveals why the FEM scheme is more robust in porous settings \cite{Schneider2017}.

Convergence rate estimates rely on the regularity of $\vecu$. The celebrating work by Li and Nirenberg \cite{Li2003a} says that under some regularity assumptions of subdomains, the solution $\vecu$ will be globally gradient bounded.
\begin{proposition}[cf. \cite{Li2003a}] \label{prop:FEM reg}
Let the subdomains $D_0,D_1,\dots,D_M$ satisfy several regularity assumptions and $\vecu$ be the solution of \cref{eq:varia form}. Then there exists a positive constant $C_\textnormal{reg}$ such that
\begin{equation} \label{eq:reg}
\Norm{\nabla \vecu}_{L^\infty(Y)} \leq C_\textnormal{reg} \sqrt{\matE:\matE}, \text{ and } \Norm{\vecu}_{H^2(D_l)} \leq C_\textnormal{reg} \sqrt{\matE:\matE}
\end{equation}
for any subdomain $D_l$.
\end{proposition}

Then the convergence rate estimate of $\Norm{\nablaS\Brackets{\vecu_N-\vecu}}$ is stated as follows.
\begin{theorem} \label{thm:FEM stress}
Suppose that the assumption in \cref{prop:FEM reg} holds. Let $\tenC$ and $\tenC_N$ satisfy Assumptions A-C, $\vecu_N$ be the solution of \cref{eq:FEM varia}. Then
\[
\Norm{\nablaS\Brackets{\vecu_N-\vecu}}_{L^2(Y)} \leq CN^{-1/2}\sqrt{\matE:\matE},
\]
where the positive constant $C$ is independent of $N$.
\end{theorem}
\begin{proof}
Since \cref{prop:FEM reg} tells that $\vecu$ is piecewisely regular enough, we first part $\I_N$ into two sets $\I_N'$ and $\I_N''$ defined as
\[
\I_N'=\Braces{\vecI:\exists 0\leq l\leq M, \text{ s.t. }\Delta_\vecI \subset D_l} \text{ and } \I_N''=\I_N \setminus \I_N'.
\]
The cardinality of $\I_N''$ can be controlled as $\Card\Brackets{\I_N''}\leq C_0N^{d-1}$, because $Y_\vecI$ for $\vecI\in \I_N''$ belongs to $O(N^{-1})$-width layers of interfaces of subdomains. Then, recalling Assumption C, we have $\tenC=\tenC_N$ in $Y_\vecI$ with $\vecI \in \I_N'$ and
\[
\begin{aligned}
\Norm{(\tenC-\tenC_N):\matVarep}_{L^2(Y)}^2=&\sum_{\vecI\in \I_N''}\int_{Y_\vecI}\Seminorm{(\tenC-\tenC_N):\matVarep}^2\dx \vecx\\
\leq& C(d,\Lambda'') \Norm{\nablaS\vecu+\matE}_{L^\infty(Y)}^2\Braces{N^{-d}\Card\Brackets{\I_N''}}\\
\leq& C(d,\Lambda'', C_0, C_\textnormal{reg}) \Brackets{\matE:\matE}N^{-1},
\end{aligned}
\]
where the last line follows from \cref{prop:FEM reg}. Similarly, for $\Norm{\nablaS\Brackets{\vecu-\Lop_N\vecu}}_{L^2(Y)}$
\[
\begin{aligned}
&\Norm{\nablaS\Brackets{\vecu-\Lop_N\vecu}}_{L^2(Y)}^2\\
\leq & C(d) \Brackets{\sum_{\vecI\in \I_N'}+\sum_{\vecI\in \I_N''}}\int_{Y_\vecI} \Seminorm{\nabla\Brackets{\vecu-\Lop_N\vecu}}^2\dx \vecx \\
\leq & C(d,C_\textnormal{int})\Braces{\sum_{\vecI\in \I_N'} N^{-2}\Norm{\vecu}_{H^2(\Delta_\vecI)}^2 + \sum_{\vecI\in \I_N''} \Norm{\vecu}_{H^1(\Delta_\vecI)}^2} \\
\leq & C(d,C_\textnormal{int})\Braces{N^{-2} \sum_{0\leq l \leq M} \Norm{\vecu}_{H^2(D_l)}^2 + N^{-d}\Card\Brackets{\I_N''} \Norm{\nabla \vecu}_{L^\infty(Y)}^2} \\
\leq & C(d,C_\textnormal{int}, C_\textnormal{reg})\Brackets{\matE:\matE} \Braces{N^{-2}+C_0N^{-1}}.
\end{aligned}
\]
Combining \cref{lem:FEM pri}, we complete the estimate.
\end{proof}

For the FEM scheme, we can obtain numerical effective coefficients by the formula
\[
\tenC^{\textnormal{eff},\textnormal{F}}_{N}:\matE \coloneqq N^{-d} \sum_{\vecI\in \I_N} \tenC_N^\star[\vecI]:\Brackets{\nablaS \vecu_N(\vecx_\vecI)+\matE}.
\]
Another important estimate in this section shows that $\tenC^{\textnormal{eff},\textnormal{F}}_{N}:\matE$ can approximate $\tenC^\textnormal{eff}$ with a higher rate.
\begin{theorem}
Under the same assumptions in \cref{thm:FEM stress}. Then
\[
\Seminorm{\matE:\Brackets{\tenC^\textnormal{eff}-\tenC^{\textnormal{eff},\textnormal{F}}_{N}}:\matE} \leq C N^{-1} \Brackets{\matE:\matE},
\]
where $\tenC^\textnormal{eff}$ is defined by \cref{eq:eff} and $C$ is a positive constant independent of $N$.
\end{theorem}

\begin{proof}
The proof is based on an observation that
\[
\matE:\tenC^\textnormal{eff}:\matE=\AngleBrackets{\Brackets{\nablaS \vecu+\matE}:\tenC:\Brackets{\nablaS \vecu+\matE}},
\]
which is a direct result of $\AngleBrackets{\nablaS \vecu:\tenC:\Brackets{\nablaS \vecu+\matE}}=0$. Similarly, for $\tenC^{\textnormal{eff},\textnormal{F}}_N$, recalling that $\vecu_N$ is piecewisely trilinear in each element, we have
\[
\begin{aligned}
\matE:\tenC^{\textnormal{eff},\textnormal{F}}_N:\matE&=\matE:\Braces{N^{-d} \sum_{\vecI\in \I_N} \tenC_N^\star[\vecI]:\Brackets{\nablaS \vecu_N(\vecx_\vecI)+\matE}}\\
&=\AngleBrackets{\Brackets{\nablaS \vecu_N+\matE}:\tenC_N:\Brackets{\nablaS \vecu_N+\matE}}.
\end{aligned}
\]
Then,
\[
\begin{aligned}
&\matE:\Brackets{\tenC^{\textnormal{eff},\textnormal{F}}_{N}-\tenC^\textnormal{eff}}:\matE\\
=&\AngleBrackets{\Brackets{\nablaS \vecu_N+\matE}:\tenC_N:\Brackets{\nablaS \vecu_N+\matE}}-\AngleBrackets{\Brackets{\nablaS \vecu+\matE}:\tenC:\Brackets{\nablaS \vecu+\matE}}\\
=&\AngleBrackets{\nablaS \Brackets{\vecu_N-\vecu}:\tenC_N:\nablaS \Brackets{\vecu_N-\vecu}}+2\AngleBrackets{\nablaS \Brackets{\vecu_N-\vecu}:\tenC:\Brackets{\nablaS \vecu+\matE}}\\
&+2\AngleBrackets{\nablaS \Brackets{\vecu_N-\vecu}:\Brackets{\tenC_N-\tenC}:\Brackets{\nablaS \vecu+\matE}}+\AngleBrackets{\matVarep:\Brackets{\tenC_N-\tenC}:\matVarep}\\
\coloneqq& J_1+J_2+J_3+J_4,
\end{aligned}
\]
where $\matVarep=\nablaS\vecu+\matE$ as previously. For $J_1$, we have $\Seminorm{J_1} \leq C N^{-1}(\matE:\matE)$ according to \cref{thm:FEM stress}. Meanwhile, $J_2=0$ follows from the variational form \cref{eq:varia form}. For $J_4$, applying the proof steps of \cref{thm:FEM stress}, we have $\Seminorm{J_4}\leq CN^{-1}(\matE:\matE)$ by the boundedness of $\matVarep$ and Assumption C. Then it is left to estimate $J_3$, by Young's inequality,
\[
\begin{aligned}
\Seminorm{J_3} \leq & \AngleBrackets{\nablaS \Brackets{\vecu_N-\vecu}:\nablaS\Brackets{\vecu_N-\vecu}} + \AngleBrackets{\matVarep:\Brackets{\tenC_N-\tenC}:\Brackets{\tenC_N-\tenC}:\matVarep}\\
\leq & CN^{-1}\Brackets{\matE:\matE}.
\end{aligned}
\]
We hence finish the proof.
\end{proof}

\bibliographystyle{unsrtnat}
\bibliography{refs}  






\end{document}